\documentclass[10pt]{svmult}
\usepackage{bbm}

\usepackage{mathptmx}       
\usepackage{helvet}         
\usepackage{courier}        
\usepackage{type1cm}        
\usepackage{makeidx}         
\usepackage{graphicx}        
\usepackage{multicol}        
\usepackage[bottom]{footmisc}

\usepackage{amsmath}

\numberwithin{equation}{section}
\numberwithin{definition}{section}
\numberwithin{theorem}{section}
\numberwithin{remark}{section}
\numberwithin{example}{section}

\DeclareSymbolFont{EULERSCRIPT}{U}{eus}{m}{n}
\SetSymbolFont{EULERSCRIPT}{bold}{U}{eus}{b}{n}
\DeclareSymbolFontAlphabet\MATHCAL{EULERSCRIPT}

\DeclareSymbolFont{AMSB}{U}{msb}{m}{n}
\DeclareSymbolFontAlphabet{\MATHBB}{AMSB}

\newcommand{\MATHSCR}{\mathcal}

\newcommand{\TEXT}{\textrm}

\newcommand{\MATHRM}{\mathrm}

\def\RR{\MATHBB{R}}       
\def\CC{\MATHBB{C}}       
\def\Rd{\MATHBB{R}^d}     

\def\vx{\mathbf{x}}   
\def\vy{\mathbf{y}}   

\def\vc{\mathbf{c}}

\def\NN{\MATHBB{N}}       
\def\CC{\MATHBB{C}}       

\def\ud{\mathrm{d}}       

\def\Schwartz{\MATHCAL{S}}        

\def\SI{\MATHCAL{SI}}  

\def\Continue{\mathrm{C}(\MATHBB{R}^d)}           
\def\ContinueInf{\mathrm{C}^{\infty}(\MATHBB{R}^d)}        
\def\ContinueInfComp{\mathrm{C}_0^{\infty}(\MATHBB{R}^d)}  
\def\Lloc{\mathrm{L}_1^{loc}(\MATHBB{R}^d)} 
\def\Ltwo{\mathrm{L}_2(\MATHBB{R}^d)}       
\def\Lone{\mathrm{L}_1(\MATHBB{R}^d)}     

\def\NativePhi{\MATHCAL{N}_{\Phi}(\MATHBB{R}^d)}  
\def\NativeG{\MATHCAL{N}_{G}(\MATHBB{R}^d)}       

\newcommand{\norm}[1]{\left\|#1\right\|}  
\newcommand{\abs}[1]{\left|#1\right|}   

\newcommand{\Matlab}{{\sc Matlab}}

\def\vP{\mathbf{P}}            
\def\vQ{\mathbf{Q}}            
\def\vB{\mathbf{B}}            

\def\FvP{\hat{\mathbf{p}}}           

\def\Hp{\mathrm{H}_{\mathbf{P}}(\MATHBB{R}^d)} 

\def\Leb{\mathrm{L}}
\def\Hil{\MATHSCR{H}}
\def\Cont{\mathrm{C}}
\def\Native{\MATHCAL{N}}
\def\HP{\mathrm{H}_{\mathbf{P}}}

\def\Real{\mathrm{Re}}
\def\RealLloc{\mathrm{Re}(\mathrm{L}_{1}^{loc}(\MATHBB{R}^d))}  

\def\RealContinue{\mathrm{Re}(\mathrm{C}(\MATHBB{R}^d))}   
\def\RealLtwo{\mathrm{Re}(\mathrm{L}_2(\MATHBB{R}^d))}       

\def\RealLone{\mathrm{Re}(\mathrm{L}_1(\MATHBB{R}^d))}       

\begin{document}

\title*{Reproducing Kernels of Generalized Sobolev Spaces via a Green Function Approach with Differential Operators}
\titlerunning{Reproducing Kernels via a Green Function Approach}

\author{Qi Ye}

\institute{Qi Ye \at Department of Applied Mathematics, Illinois
Institute of Technology \email{qye3@iit.edu}}



\maketitle

\abstract{ In this paper we introduce a generalization of the
classical $\Leb_2(\Rd)$-based Sobolev spaces with the help of a
vector differential operator $\vP$ which consists of finitely or
countably many differential operators $P_n$ which themselves are
linear combinations of distributional derivatives. We find that certain proper full-space
Green functions $G$ with respect to $L=\vP^{\ast T}\vP$ are positive
definite functions. Here we ensure that the vector distributional
adjoint operator $\vP^{\ast}$ of $\vP$ is well-defined in the
distributional sense. We then provide sufficient conditions under
which our generalized Sobolev space will become a reproducing-kernel
Hilbert space whose reproducing kernel can be computed via the
associated Green function $G$. As an application of this theoretical
framework we use $G$ to construct multivariate minimum-norm
interpolants $s_{f,X}$ to data sampled from a generalized Sobolev
function $f$ on $X$. Among other examples we show the
reproducing-kernel Hilbert space of the Gaussian function is
equivalent to a generalized Sobolev space.}




\section{Introduction}

There is a steadily increasing body of literature on radial basis
function and other kernel-based approximation methods in such
application areas as scattered data approximation, statistical or
machine learning and numerical solutions of partial differential
equations. Some recent books and survey papers on kernel methods are
\cite{BerThAg04, Buh03, Fas07, KBU02ab, SchWen06, SchSmo02, Wah90,
Wen05}. Given a set of data sites $X$ and associated values $Y$ sampled from a continuous function
$f$, we use a positive definite function $\Phi$ to set up an
interpolant $s_{f,X}$ to approximate $f$. It is well-known that within the reproducing-kernel Hilbert space framework one can also discuss the error analysis and optimal recovery of the
interpolation process whenever $f$ belongs to the related
reproducing-kernel Hilbert space $\NativePhi$
(see~Section~\ref{s:PD-RKHS}~and~\cite{Wen05}). However, there are still
a number of difficulties and challenges associated with this method. Two important
questions still in need of a satisfactory answer are:
\emph{What kind of functions belong to a given reproducing-kernel
Hilbert space?} and \emph{Which kernel function should we utilize for a
particular application?}

In the survey paper \cite{SchWen06} the authors give the reader some
guidance for dealing with this problem. Others -- especially
statisticians -- will attempt to find the ``best'' kernel function
by selecting an ``optimal'' scaling parameter. Here such techniques
as cross-validation and maximum likelihood estimation are often
mentioned (see, e.g. \cite{Ste99, Wah90}). The recent
paper~\cite{DevRon10} derives a new error bound of the approximation
by the fundamental functions (Green functions) using scattered
shifts.

In \cite{Sch08} the author noted that many kernels can be embedded into the
classical Sobolev spaces, but that is was not so clear what the difference between the various kernels was.
As a possible approach to this, he suggested that one might scale the classical Sobolev space with the help of a scaling parameter integrated into the kernel function. Once this is done, we can choose the
``best'' scaling parameter to set up the interpolant that minimizes the
norm of the error functional for a given set of scattered data.
Examples~\ref{ex:Sobolevspline}
and~\ref{ex:Matern} will demonstrate that the classical Sobolev space can be
reconstructed by different inner-products that allow us to balance
the required derivatives by selecting various scaling parameters. The related
Green functions are the Sobolev splines (Mat\'ern functions) with
appropriate scaling parameters (see \cite{Ste99}). Finally,
Example~\ref{ex:Gaussian} shows that the reproducing-kernel Hilbert
space of the Gaussian function is also equivalent to a generalized
Sobolev space. This generalized Sobolev space has been applied in the context of
support vector machines and the study of motion coherence
(see~\cite{SSM98, YuiGrz88}).

However, in this paper, we view this problem in a slightly different
way. We hope to construct the reproducing kernel and the
reproducing-kernel Hilbert space with the help of countably many differential
operators $\left\{P_n\right\}_{n=1}^{\infty}$ which are themselves linear
combinations of distributional derivatives (defined in
Section~\ref{s:distributions}). Handling the vector differential
operator $\vP:=\left(P_1,\cdots,P_n,\cdots\right)^T$, we will
generalize the concept of classical real-valued $\Leb_2(\Rd)$-based
Sobolev space $\Hil^n(\Rd)$ to be a generalized Sobolev space $\Hp$
with a semi-inner product (see Definition~\ref{d:GenSolSp}).

In the following we use the notation $\Real(\MATHCAL{E})$ to denote
the collection of all real-valued functions of the function space
$\MATHCAL{E}$. For example, $\RealContinue$ expresses the collection
of all real-valued continuous functions on $\Rd$.
%
The real classical Sobolev space is usually defined as
\[
\Hil^n(\Rd) := \left\{ f \in\RealLtwo: D^{\alpha} f\in\Ltwo\TEXT{
for all }\abs{\alpha}\leq n,\alpha\in\NN^d_0\right\}
\]
with inner product
\[
(f,g)_{\Hil^n(\Rd)} := \sum_{\abs{\alpha}\leq n}(D^\alpha f,
\overline{D^\alpha g}), \quad{ }f,g\in\Hil^n(\Rd),
\]
where $(\cdot,\overline{~\cdot})$ is the standard $\Ltwo$-inner
product. Our real generalized Sobolev space will be of a very
similar form, namely
\[
\Hp:= \Big\{ f\in\RealLtwo:\left\{P_n
f\right\}_{n=1}^{\infty}\subseteq\Ltwo\TEXT{ and }
\sum_{n=1}^{\infty}\norm{P_nf}_{\Ltwo}^2<\infty \Big\}
\]
with the semi-inner product
\[
(f,g)_{\Hp}:=\sum_{n=1}^{\infty}(P_n f,\overline{P_n g}), \quad{
}f,g\in\Hp.
\]

We further wish to connect the reproducing kernel to a (full-space)
Green function $G$ with respect to some differential operator $L$,
i.e., $LG=\delta_0$. Since the Dirac delta function $\delta_0$ at
the origin is just a tempered distribution in the dual space of the
Schwartz space, the Green function should be regarded as a tempered
distribution as well. We find that $L$ can be computed by a vector
differential operator $\vP:=\left(P_1,\cdots,P_n\right)^T$ and its
distributional adjoint operator
$\vP^{\ast}:=\big(P_1^{\ast},\cdots,P_n^{\ast}\big)^T$, i.e.,
$L=\vP^{\ast T}\vP=\sum_{j=1}^nP_j^{\ast}P_j$. The definition of the
distributional adjoint operator is given in
Section~\ref{s:distributions}. Under some sufficient conditions, we
can prove that the Green function $G$ is a positive definite
function according to Theorem~\ref{t:GF-PD}. In that case the
reproducing-kernel Hilbert space $\NativeG$ related to $G$ is
well-defined in Section~\ref{s:PD-RKHS} and~\cite{Wen05}, and its
reproducing kernel $K$ on $\Rd$ is equal to
$K(\vx,\vy):=G(\vx-\vy)$. Moreover, Theorems~\ref{t:RKHS-Hp} shows
that $\NativeG$ is equivalent to $\Hp$ for some simple additional
conditions. In the proof, we use techniques of distributional
Fourier transforms defined in Section~\ref{s:DistFT}
and~\cite{SteWei75}, as well as the classical Fourier transform.
%


\section{Background}\label{s:background}

Given data sites $X=\{\vx_1,\cdots,\vx_N\}\subset\Rd$ (which we also
identify with the centers of our kernel functions below) and values
$Y=\{y_1,\cdots,y_N\}\subset\RR$ of a real-valued continuous function
$f$ on $X$, we wish to approximate this function $f$ by a linear
combination of translates of a positive definite (see
Section~\ref{s:PD}) function $\Phi$.

To this end we set up the interpolant in the form
\begin{equation}\label{interpolant}
s_{f,X}(\vx):=\sum_{j=1}^{N}c_j\Phi(\vx-\vx_j),
\quad{}\vx\in\Rd,
\end{equation}
and require it to satisfy the additional interpolation conditions
\begin{equation}\label{interpolation_conditions}
s_{f,X}(\vx_j)=y_j,           \quad{  }j=1,\ldots,N.
\end{equation}

The above system (\ref{interpolation_conditions}) is equivalent
to a uniquely solvable linear system
\[
\mathbf{A}_{\Phi,X}\vc=\mathbf{Y},
\]
where
$\mathbf{A}_{\Phi,X}:=\left(\Phi(\vx_j-\vx_k)\right)_{j,k=1}^{N,N}\in
\RR^{N\times N}$, $\vc:=(c_1,\cdots,c_N)^T$ and
$\mathbf{Y}:=(y_1,\cdots,y_N)^T$. All of the above is discussed in
detail in \cite[Chapter~6.1]{Wen05}.

\begin{example}\label{ex:unisobolevspline}
One of the best known examples that fits into this framework is the \emph{univariate Sobolev spline (Mat\'{e}rn)} interpolant
\[
s_{f,X}(\MATHRM{x}):=\sum_{j=1}^N c_j\Phi(\MATHRM{x}-\MATHRM{x}_j),\quad{}\Phi(\MATHRM{x}):=\frac{1}{2\sigma}\exp(-\sigma\abs{\MATHRM{x}}),
\quad{}\MATHRM{x}\in\RR,
\]
and $\sigma >0$ is called the \emph{Sobolev parameter}.

We can check that $\Phi$ is a positive definite function. Moreover,
if we define $\vP = (P_1,P_2)^T := (d/d \MATHRM{x}, \sigma I)^T$ and
$\vP^{\ast} = (P_1^{\ast},P_2^{\ast})^T := (-d/d \MATHRM{x}, \sigma I)^T$,
then $\Phi$ is the Green function with respect to $L = \vP^{\ast
T}\vP = -d^2/d \MATHRM{x}^2+\sigma^2I$.

Furthermore, the interpolant $s_{f,X}$ from
(\ref{interpolant})-(\ref{interpolation_conditions}) minimizes the
norm $\norm{\cdot}_{Sob}$ of all $f \in \Cont^1(\RR)\cap\RealLtwo$
with
\[
\norm{f}_{Sob}^2:= (\vP f,\overline{\vP f})_{Sob} =
\int_{\RR}\abs{\frac{d f}{d \MATHRM{x}}(\MATHRM{x})}^2\ud \MATHRM{x}+\int_{\RR}\sigma^2
\abs{f(\MATHRM{x})}^2\ud \MATHRM{x}<\infty
\]
subject to the constraints $s_{f,X}(\MATHRM{x}_i)=y_j$,
$j=1,\ldots,N$ (see~\cite{BerThAg04}).
\end{example}

As we will show in Section~\ref{s:RKHS-HS}, we can in general
construct a reproducing-kernel Hilbert space $\NativePhi$ from
a positive definite function $\Phi$ such that the interpolant $s_{f,X}$ is the best approximation
of the function $f$ in $\NativePhi$ fitting the values $Y$ on the data sites $X$.


\section{Positive Definite Functions and Reproducing-Kernel Hilbert
Spaces}\label{s:PD-RKHS}

Most of the material presented in this section can be found in the excellent monograph \cite{Wen05}. For the reader's convenience we repeat here what is essential to our discussion later on.

\subsection{Positive Definite Functions}\label{s:PD}

\begin{definition}[{\cite[Definition~6.1]{Wen05}}]
A continuous even function $\Phi:\Rd\to\RR$ is said to be a
\emph{positive definite function} if, for all $N\in\NN$, all
pairwise distinct centers $\vx_1,\ldots,\vx_N\in \Rd$, and all
$\vc=(c_1,\cdots,c_N)^T\in\RR^N\setminus \{0\}$, the quadratic form
\[
\sum_{j=1}^N\sum_{k=1}^{N}c_jc_k\Phi(\vx_j-\vx_k)>0.
\]
\end{definition}

\begin{theorem}[{\cite[Theorem~6.11]{Wen05}}]\label{t:PD-FT}
Suppose that $\Phi\in\RealContinue\cap\Lone$ is an even function and
its $\Lone$-Fourier transform is $\hat{\phi}$.
Then $\Phi$ is positive definite if and
only if $\Phi$ is bounded and $\hat{\phi}$ is nonnegative and nonvanishing.
\end{theorem}

%
%

\subsection{Reproducing-Kernel Hilbert
Spaces}\label{s:RKHS-HS}

\begin{definition}[{\cite[Definition~10.1]{Wen05}}]
Let $\mathrm{H}(\Rd)$ be a real Hilbert space of functions
$f:\Rd\rightarrow\RR$. A kernel $K:\Rd\times\Rd\rightarrow\RR$ is
called a \emph{reproducing kernel} for $\mathrm{H}(\Rd)$ if
\begin{eqnarray*}
&&(i)~K(\cdot,\vy)\in\mathrm{H}(\Rd),~\TEXT{for all}~\vy\in\Rd,\\
&&(ii)~f(\vy)=\left(K(\cdot,\vy),f\right)_{\mathrm{H}(\Rd)},~\TEXT{for
all}~f\in\mathrm{H}(\Rd)~\TEXT{and all}~\vy\in\Rd.
\end{eqnarray*}
In this case $\mathrm{H}(\Rd)$ is called a \emph{reproducing-kernel Hilbert space}.
\end{definition}

\begin{theorem}[{\cite[Theorem~10.12]{Wen05}}]\label{t:RKHS-FT}
Suppose that $\Phi\in\Continue\cap\RealLone$ is an even function and
its $\Lone$-Fourier transform is $\hat{\phi}$. If $\Phi$ is a
positive definite function, then
\[
\NativePhi=\left\{f\in\Continue\cap\RealLtwo:\hat{\phi}^{-1/2}\hat{f}\in\Ltwo\right\},
\]
is a reproducing-kernel Hilbert space with reproducing kernel
\[
K(\vx,\vy)=\Phi(\vx-\vy),\quad{}\vx,\vy\in\Rd,
\]
and its inner product satisfies
\[
(f,g)_{\NativePhi}=(2\pi)^{-d/2}\int_{\Rd}\frac{\hat{f}(\vx)\overline{\hat{g}(\vx)}}{\hat{\phi}(\vx)}\ud\vx,
\quad{}f,g\in\NativePhi,
\]
where $\hat{f}$ and $\hat{g}$ are the $\Ltwo$-Fourier transform of
$f$ and $g$, respectively.
\end{theorem}


\section{Connecting Green Functions and Generalized Sobolev Spaces to Positive Definite Functions and Reproducing-Kernel Hilbert Spaces}\label{s:GenSob-RKHSG}

\subsection{Differential Operators and Distributional Adjoint Operators}\label{s:distributions}

First, we define a metric $\rho$ on the Schwartz space
\[
\Schwartz:=\Big\{
\gamma\in\ContinueInf:\forall\alpha,\beta\in\NN_0^d,\exists
C_{\alpha\beta\gamma}>0\TEXT{ s.t. }\sup_{\vx\in\Rd}\abs{
\vx^{\beta}D^{\alpha}\gamma(\vx)}\leq C_{\alpha\beta\gamma}\Big\}
\]
so that it becomes a Fr\'{e}chet space, where the metric $\rho$ is
given by
\[
\rho(\gamma_1,\gamma_2):=\sum_{\alpha,\beta\in\NN_0^d}
2^{-\abs{\alpha}-\abs{\beta}}\frac{\rho_{\alpha\beta}(\gamma_1-\gamma_2)}{1+\rho_{\alpha\beta}(\gamma_1-\gamma_2)},
\quad{}\rho_{\alpha\beta}(\gamma):=\sup_{\vx\in\Rd}\abs{\vx^{\beta}D^{\alpha}\gamma(\vx)},
\]
for each $\gamma_1,\gamma_2,\gamma\in\Schwartz$. This means that a
sequence $\left\{\gamma_n\right\}_{n=1}^{\infty}$ of $\Schwartz$
converges to an element $\gamma\in\Schwartz$ if and only if
$\vx^{\beta}D^{\alpha}\gamma_n(\vx)$ converges uniformly to
$\vx^{\beta}D^{\alpha}\gamma(\vx)$ on $\Rd$ for each
$\alpha,\beta\in\NN_0^d$. Together with its metric $\rho$ the Schwartz space $\Schwartz$ is
regarded as the classical test function space.

Let $\Schwartz'$ be the associated space of tempered
distributions (dual space of $\Schwartz$ or space of continuous linear
functionals on $\Schwartz$). We introduce the notation
\[
\langle T,\gamma\rangle:=T(\gamma),\quad\TEXT{for each
}T\in\Schwartz' \TEXT{ and }\gamma\in\Schwartz.
\]
For any $f,g\in\Lloc$ whose product $fg$ is integrable on $\Rd$
we denote a bilinear form by
\[
(f,g):=\int_{\Rd}f(\vx)g(\vx)\ud\vx.
\]
If $f,g\in\Ltwo$ then $(f,\overline{g})$ is equal to the standard
$\Ltwo$-inner product.

For each $f\in\Lloc\cap\SI$ there exists a
unique tempered distribution $T_f\in\Schwartz'$ such that
\[
\langle T_f,\gamma\rangle =(f,\gamma),\quad\TEXT{for each }
\gamma\in\Schwartz.
\]
So $f\in\Lloc\cap\SI$ can be viewed as an element of $\Schwartz'$
and we rewrite $T_f:=f$. This means that $\Lloc\cap\SI$ can be
embedded into $\Schwartz'$, i.e., $\Lloc\cap\SI\subseteq\Schwartz'$.
The Dirac delta function (Dirac distribution) $\delta_0$
concentrated at the origin is also an element of $\Schwartz'$, i.e.,
$\langle \delta_0,\gamma\rangle =\gamma(0)$ for each
$\gamma\in\Schwartz$ (see \cite[Chapter~1]{SteWei75}
and~\cite[Chapter~11]{HunNac05}).
\begin{remark}
$\SI$ denotes the collection of \emph{slowly increasing functions}
which grow at most like any particular fixed polynomial, i.e.,
\[
\SI:=\left\{ f\in\Rd\rightarrow\CC:
f(\vx)=\MATHCAL{O}(\norm{\vx}_2^{m})\TEXT{ as
}\norm{\vx_2}\to\infty\TEXT{ for some }m\in\NN_0 \right\},
\]
where the notation $f=\MATHCAL{O}(g)$ means that there is a positive
number $M$ such that $\abs{f}\leq M\abs{g}$.
\end{remark}

First, we will define the \emph{distributional derivative}
$P:\Schwartz'\rightarrow\Schwartz'$ by (strong) derivative
\[
D^{\alpha}:=\prod_{k=1}^{d}\frac{\partial^{\alpha_k}}{\partial
x_k^{\alpha_k}}, \quad{}\TEXT{where
}\abs{\alpha}:=\sum_{k=1}^d\alpha_k,~
\alpha:=\left(\alpha_1,\cdots,\alpha_d\right)^T\in\NN_0^d.
\]
According to \cite[Chapter~1]{SteWei75}, the derivative $D^{\alpha}$
is a continuous linear operator from $\Schwartz$ into $\Schwartz$.
Then $P$ can be well-defined by
\[
\langle PT,\gamma \rangle:=(-1)^{\abs{\alpha}}\langle
T,D^{\alpha}\gamma \rangle,\quad{}\TEXT{for each }T\in\Schwartz'
\TEXT{ and }\gamma\in\Schwartz.
\]
Since
$(D^{\alpha}\gamma_1,\gamma_2)=(-1)^{\abs{\alpha}}(\gamma_1,D^{\alpha}\gamma_2)$
for each $\gamma_1,\gamma_2\in\Schwartz$, the restricted operator
$P|_{\Schwartz}$ and the derivative $D^{\alpha}$ coincide on
$\Schwartz$. For convenience, we denote the distributional
derivative as $P:=D^{\alpha}$ (similar
as~\cite[Chapter~1]{AdaFou03}).

Next we wish to define the distributional adjoint operator of the
distributional derivative. In the same way as before we can
introduce the linear operator
$P^{\ast}:\Schwartz'\rightarrow\Schwartz'$ by using the derivative
$(-1)^{\abs{\alpha}}D^{\alpha}$, i.e.,
\[
\langle P^{\ast}T,\gamma \rangle:=(-1)^{\abs{\alpha}}\langle
T,(-1)^{\abs{\alpha}}D^{\alpha}\gamma \rangle=\langle
T,D^{\alpha}\gamma \rangle,\quad{}\TEXT{for each }T\in\Schwartz'
\TEXT{ and }\gamma\in\Schwartz.
\]
Moreover, the restricted operator $P^{\ast}|_{\Schwartz}$ and the
derivative $(-1)^{\abs{\alpha}}D^{\alpha}$ coincide on $\Schwartz$.
We call the linear operator $P^{\ast}$ the \emph{distributional
adjoint operator} of the distributional derivative $P=D^{\alpha}$.
It can be written as
$P^{\ast}=(-1)^{\abs{\alpha}}P=(-1)^{\abs{\alpha}}D^{\alpha}$ also.
Usually people call the derivative
$P^{\ast}|_{\Schwartz}=(-1)^{\abs{\alpha}}D^{\alpha}:\Schwartz\rightarrow\Schwartz$
as the classical adjoint operator of the distributional derivative
$P=D^{\alpha}:\Schwartz'\rightarrow\Schwartz'$. Here we can think of
the classical adjoint operator $P^{\ast}|_{\Schwartz}$ extended to
be the distributional adjoint operator $P^{\ast}$.

Now we will define a more general differential operator and its
distributional adjoint operator by using real linear combinations of
distributional derivatives.
\begin{definition}\label{d:DistrAdjoint}
The \emph{differential operator} (with real-valued constant
coefficients) \\ $P:\Schwartz'\rightarrow\Schwartz'$ is defined as
\[
P:=\sum_{\abs{\alpha}\leq m}c_{\alpha}D^{\alpha},\quad{}\TEXT{where
}c_{\alpha}\in\RR\TEXT{ and }\alpha\in\NN_0^d,~m\in\NN_0.
\]
Its \emph{distributional adjoint operator}
$P^{\ast}:\Schwartz'\rightarrow\Schwartz'$ is well-defined by
\[
P^{\ast}:=\sum_{\abs{\alpha}\leq
m}(-1)^{\abs{\alpha}}c_{\alpha}D^{\alpha}.
\]
To streamline terminology we will refer to differential operators
(with real-valued constant coefficients) and distributional adjoint
operators simply as differential operators and adjoint operators
respectively in this article.
\end{definition}

It is obvious that the adjoint operator $P^{\ast}$ of the
differential operator $P$ is also a differential operator. We can
further check that the differential operator $P$ and its adjoint
operator $P^{\ast}$ have the following property
\[
\langle PT,\gamma\rangle=\langle T,P^{\ast}\gamma\rangle~\TEXT{ and
}~\langle P^{\ast}T,\gamma\rangle=\langle T,P\gamma\rangle,
\quad\TEXT{for each }T\in\Schwartz' \TEXT{ and }\gamma\in\Schwartz.
\]
Any differential operator $P$ is also \emph{complex-adjoint
invariant}, i.e.,
\[
\overline{P\gamma}=P\overline{\gamma},\quad\TEXT{for each
}\gamma\in\Schwartz.
\]
%
\begin{remark}
Our distributional adjoint operator is different from the usual
adjoint operators of bounded linear operators defined in Hilbert or
Banach space. Our operator is formed in the dual space of the
Schwartz space and it may be not a bounded operator if $\Schwartz'$
is defined as a metric space. But it
is continuous when $\Schwartz'$ is given the weak-star topology as
the dual of $\Schwartz$. However, the idea of this construction is
similar to the classical ones. Therefore we call it an adjoint as
well.
\end{remark}
%

Finally, we know that the classical Sobolev spaces are defined by
weak derivatives. Here we will explain some relationships
between the distributional derivative $P:=D^{\alpha}$ and the
$\alpha^{\textrm{th}}$ weak derivative. Fixing any $f\in\Lloc\cap\SI$, if there
is a generalized function $g\in\Lloc\cap\SI$ such that
\[
(g,\gamma)=(-1)^{\abs{\alpha}}(f,D^{\alpha}\gamma)=\langle
D^{\alpha}f,\gamma\rangle,\quad\TEXT{for each }\gamma\in\Schwartz,
\]
then $D^{\alpha}f=g$ is called the \emph{$\alpha^{\textrm{th}}$ weak
derivative of $f$}. This means that distributional derivatives and
weak derivatives are the same on the classical Sobolev space.
\begin{remark}\label{r:WeakDeriv}
In the book~\cite{AdaFou03}, the definition of the weak derivative
has tiny differences from the one we use in this article. In particular, they use different test
functions to derive the weak derivative. We now state their
definition in order to compare it to the above mentioned weak derivative. Fixing any
$f\in\Lloc$, if there is a generalized function $g\in\Lloc$ such
that
\[
(g,\gamma)=(-1)^{\abs{\alpha}}(f,D^{\alpha}\gamma),\quad\TEXT{for
each }\gamma\in\MATHSCR{D}:=\ContinueInfComp,
\]
then they call $D^{\alpha}f:=g$ the \emph{$\alpha^{\textrm{th}}$
weak derivative of $f$}. However, if $f\in\Lloc\cap\SI$ and
$D^{\alpha}f\in\Lloc\cap\SI$, then the two definitions of weak
derivatives are equivalent since $\MATHSCR{D}(\Rd)$ is dense in
$\Schwartz$, where $\MATHSCR{D}(\Rd)=\Cont^{\infty}_0(\Rd)$ and its
dual space $\MATHSCR{D}(\Rd)'$ are defined in \cite[Chapter~1]{AdaFou03}. In this case, we can consider the two
weak derivatives as being the same.
\end{remark}
%

%
%

\subsection{Distributional Fourier Transforms}\label{s:DistFT}

Denote the \emph{Fourier transform} and \emph{inverse Fourier transform}
of any $\gamma\in\Schwartz$ by
\[
\hat{\gamma}(\vx):=(2\pi)^{-d/2}\int_{\Rd}\gamma(\vy)e^{-i\vx^T\vy}\ud\vy,~
\check{\gamma}(\vx):=(2\pi)^{-d/2}\int_{\Rd}\gamma(\vy)e^{i\vx^T\vy}\ud\vy,~i:=\sqrt{-1}.
\]
Since $\hat{\gamma}$ belongs to $\Schwartz$ for each
$\gamma\in\Schwartz$ and the Fourier transform map is a
homeomorphism of $\Schwartz$ onto itself, the \emph{distributional
Fourier transform} $\hat{T}\in\Schwartz'$ of the tempered distribution
$T\in\Schwartz'$ is well-defined by
\[
\langle \hat{T},\gamma\rangle :=\langle T,\hat{\gamma}\rangle,
\quad\TEXT{for each }\gamma\in\Schwartz.
\]
Since $\overline{\hat{\gamma}}=\check{\overline{\gamma}}$ for each
$\gamma\in\Schwartz$, we have
\[
\langle T,\overline{\gamma}\rangle =\langle
\hat{T},\overline{\hat{\gamma}}\rangle,\quad\TEXT{for each
}T\in\Schwartz' \TEXT{ and }\gamma\in\Schwartz.
\]

So the Fourier transform of $\gamma\in\Schwartz$ is the same as its
distributional transform. If $f\in\Lone$ or $f\in\Ltwo$ then its
$\Lone$-Fourier transform or $\Ltwo$-Fourier transform is equal to
its distributional Fourier transform. The distributional Fourier
transform $\hat{\delta}_0$ of the Dirac delta function $\delta_0$ is
equal to $(2\pi)^{-d/2}$. Moreover, we can also check that the
distributional Fourier transform map is an isomorphism of the
topological vector space $\Schwartz'$ onto itself (see \cite[Chapter
1]{SteWei75} and \cite[Chapter~7]{AdaFou03}).

Now we want to define the distributional Fourier transforms of
differential operators. First, we will derive a linear operator
$\mathcal{L}:\Schwartz'\rightarrow\Schwartz'$ using any fixed
complex-valued polynomial $\hat{p}$ on $\Rd$. Since the linear
operator $\gamma\mapsto \hat{p}\gamma$ is a continuous operator from
$\Schwartz$ into $\Schwartz$, the linear operator $\mathcal{L}$ can
be well-defined by
\[
\langle \mathcal{L}T,\gamma\rangle:=\langle
T,\hat{p}\gamma\rangle,\quad\TEXT{for each }T\in\Schwartz' \TEXT{
and }\gamma\in\Schwartz.
\]
Since $\mathcal{L}f=\hat{p}f\in\Lloc\cap\SI$ when $f\in\Lloc\in\SI$,
we use the notation $\mathcal{L}:=\hat{p}$ for convenience.

Next we consider the case of distributional derivatives. According
to \cite[Chapter~1]{SteWei75}, we know that
$\widehat{D^{\alpha}\gamma}=\hat{p}\hat{\gamma}$ for each
$\gamma\in\Schwartz$, where $\hat{p}(\vx):=(i\vx)^{\alpha}$ is a
complex-valued polynomial on $\Rd$. Hence we can verify that
$\langle \widehat{PT},\gamma \rangle=\langle
T,\hat{p}\hat{\gamma}\rangle$ for each $T\in\Schwartz'$ and
$\gamma\in\Schwartz$, where $P:=D^{\alpha}$ and
$\hat{p}(\vx):=(i\vx)^{\alpha}$. This implies that
$\widehat{PT}=F\hat{T}=\hat{p}\hat{T}$ for each $T\in\Schwartz'$.
Therefore we can denote the distributional Fourier transform of a
differential operator in the following way.
\begin{definition}\label{d:DistrFourier}
Let $P$ be a differential operator. If there is a complex-valued
polynomial $\hat{p}$ on $\Rd$ such that
\[
\langle \widehat{PT},\gamma\rangle = \langle
\hat{T},\hat{p}\gamma\rangle=\langle \hat{p}\hat{T},\gamma\rangle
\quad\TEXT{for each }T\in\Schwartz' \TEXT{ and }\gamma\in\Schwartz,
\]
then $\hat{p}$ is said to be the \emph{distributional Fourier
transform} of $P$.
\end{definition}

According to Definition~\ref{d:DistrFourier}, each differential
operator $P$ possesses a distributional Fourier transform $\hat{p}$.
The complex-valued polynomial $\hat{p}$ can be written explicitly as
\[
\hat{p}(\vx)=\sum_{\abs{\alpha}\leq
m}c_{\alpha}(i\vx)^{\alpha},\quad{}\TEXT{where
}P=\sum_{\abs{\alpha}\leq m}c_{\alpha}D^{\alpha},\
c_{\alpha}\in\RR,\ \alpha\in\NN_0^d,\ m\in\NN_0.
\]
Moreover, since $P$ is defined with the real-valued constant
coefficients, we have $\hat{p}^{\ast}=\overline{\hat{p}}$, where
$\hat{p}^{\ast}$ is the distributional Fourier transform of the
adjoint operator $P^{\ast}$.

%
%

\subsection{Green Functions and Generalized Sobolev Spaces}

\begin{definition}
$G$ is the (full-space) \emph{Green function with respect to the
differential operator $L$} if $G\in\Schwartz'$ satisfies the
equation
\begin{equation}\label{Green}
LG=\delta_0.
\end{equation}
\end{definition}
Equation~(\ref{Green}) is to be understood in the distributional sense which means that
$\langle G,L^{\ast}\gamma \rangle=\langle LG,\gamma \rangle=\langle
\delta_0,\gamma \rangle=\gamma(0)$ for each $\gamma\in\Schwartz$.

According to Theorem~\ref{t:PD-FT} we can obtain the following
theorem.

\begin{theorem}\label{t:GF-PD}
Let $L$ be a differential operator and its Fourier transform
$\hat{l}$ be positive on $\Rd$ so that $\hat{l}^{-1}\in\Lone$. If
the Green function $G\in\Continue\cap\RealLone$ with respect to $L$
is an even function, then $G$ is a positive definite function and
\[
\hat{\MATHCAL{G}}(\vx):=(2\pi)^{-d/2}\hat{l}(\vx)^{-1},\quad{}\vx\in\Rd.
\]
is the $\Lone$-Fourier transform of $G$.
\end{theorem}
\begin{proof}
First we want to prove that $\hat{\MATHCAL{G}}$ is the
$\Lone$-Fourier transform of $G$. The fact that
$\hat{l}^{-1}\in\Lone$ implies that $\hat{\MATHCAL{G}}\in\Lone$. If
we can check that
\[
\langle
\widehat{G},\gamma\rangle=(\hat{\MATHCAL{G}},\gamma),\quad{}\TEXT{for
each }\gamma\in\Schwartz,
\]
where $\widehat{G}$ is the distributional Fourier transform of $G$,
then we can conclude that $\hat{\MATHCAL{G}}$ is the $\Lone$-Fourier
transform of $G$. Since $\hat{l}^{-1}\in\ContinueInf$ and
$D^{\alpha}\left(\hat{l}^{-1}\right)\in\SI$ for each
$\alpha\in\NN_0^d$, $\hat{l}^{-1}\gamma\in\Schwartz$ for each
$\gamma\in\Schwartz$. Hence
\begin{eqnarray*}
&&\langle \widehat{G},\gamma\rangle=\langle
\hat{l}\widehat{G},\hat{l}^{-1}\gamma\rangle= \langle
\widehat{LG},\hat{l}^{-1}\gamma\rangle= \langle
\hat{\delta_0},\hat{l}^{-1}\gamma\rangle\\
&=& \langle (2\pi)^{-d/2},\hat{l}^{-1}\gamma\rangle=
((2\pi)^{-d/2},\hat{l}^{-1}\gamma)=(\hat{\MATHCAL{G}},\gamma).
\end{eqnarray*}

According to \cite[Corollary~5.24]{Wen05}, $G$ can be recovered from
its $\Lone$-Fourier transform, i.e.,
\[
G(\vx)=(2\pi)^{-d/2}\int_{\Rd}\hat{\MATHCAL{G}}(\vy)e^{i\vx^T\vy}\ud\vy,\quad{}\vx\in\Rd.
\]
Then we have
\[
\abs{G(\vx)}\leq(2\pi)^{-d}\int_{\Rd}\abs{l(\vy)^{-1}e^{i\vx^T\vy}\ud\vy}
\leq(2\pi)^{-d}\norm{l^{-1}}_{\Lone}<\infty,
\]
which shows that $G$ is bounded. Since $\hat{\MATHCAL{G}}$ is
positive on $\Rd$, $G$ is a positive definite function by
Theorem~\ref{t:PD-FT}.
\end{proof}

%
\begin{definition}\label{d:GenSolSp}
Let the vector differential operator  $\vP=\left(P_1, \cdots ,P_n,
\cdots \right)^T$ be set up by countably many differential operators
$\left\{P_n\right\}_{n=1}^{\infty}$. The real \emph{generalized
Sobolev space} induced by $\vP$ is defined by
\[
\Hp:= \Big\{ f\in\RealLtwo:\left\{P_n
f\right\}_{n=1}^{\infty}\subseteq\Ltwo\TEXT{ and }
\sum_{n=1}^{\infty}\norm{P_nf}_{\Ltwo}^2<\infty \Big\}
\]
and it is equipped with the semi-inner product
\[
(f,g)_{\Hp}:=\sum_{n=1}^{\infty}(P_n f,\overline{P_n g}), \quad{
}f,g\in\Hp.
\]
\end{definition}
What is the meaning of $\Hp$? By the definition of the generalized
Sobolev space, we know that $\Hp$ is a real-valued subspace of
$\Ltwo$ and it is equipped with a semi-inner product induced by the
vector differential operator $\vP$. On the other hand,
$f\in\RealLtwo$ belongs to $\Hp$ if and only if there is a sequence
$\{g_n\}_{n=1}^{\infty}\subset\Ltwo$ such that
$\sum_{n=1}^{\infty}\norm{g_n}_{\Ltwo}^2<\infty$ and
\[
(g_n,\gamma)=\langle g_n,\gamma \rangle=\langle P_nf,\gamma
\rangle=\langle f,P^{\ast}_n\gamma
\rangle=(f,P^{\ast}_n\gamma),\quad{}\TEXT{for each
}\gamma\in\Schwartz,~n\in\NN.
\]

In the following theorems of this section we only consider $\vP$
constructed by a finite number of differential operators
$P_1,\ldots,P_n$. If $\vP:=\left(P_1, \cdots ,P_n \right)^T$, then
the differential operator
\[
L:=\vP^{\ast T}\vP=\sum_{j=1}^nP_j^{\ast}P_j
\]
is well-defined, where $\vP^{\ast}:=\left(P_1^{\ast}, \cdots
,P_n^{\ast} \right)^T$ is the adjoint operator of $\vP$.
Furthermore, the distributional Fourier transform $\hat{l}$ of $L$
can be computed as the form
\[
\hat{l}(\vx)=\sum_{j=1}^n\hat{p}_j^{\ast}(\vx)\hat{p}_j(\vx)
=\sum_{j=1}^n\overline{\hat{p}_j(\vx)}\hat{p}_j(\vx)
=\sum_{j=1}^n\abs{\hat{p}_j(\vx)}^2=\norm{\FvP(\vx)}_2^2,\quad{}\vx\in\Rd,
\]
where $\FvP=( \hat{p}_1,\cdots,\hat{p}_n )^T$ is the distributional
Fourier transform of $\vP$ and $\FvP^{\ast}=(
\hat{p}_1^{\ast}\cdots,\hat{p}_n^{\ast})^T$ is the distributional
Fourier transform of $\vP^{\ast}$. Now we can obtain the main
theorem about the space $\Hp$ induced by the vector differential operator
$\vP:=\left(P_1, \cdots ,P_n\right)^T$.

\begin{theorem}\label{t:RKHS-Hp}
Let the vector differential operator be $\vP:=\left(P_1, \cdots
,P_n\right)^T$ and its distributional Fourier transform be $\FvP:=(
\hat{p}_1,\cdots,\hat{p}_n)^T$. Suppose that $\FvP$ is nonzero on
$\Rd$ and $\vx\mapsto\norm{\FvP(\vx)}_2^{-2}\in\Lone$. If the Green
function $G\in\Continue\cap\RealLone$ with respect to the
differential operator $L=\vP^{\ast T}\vP$ is chosen to be an even
function, then $G$ is a positive definite function and its related
reproducing-kernel Hilbert space $\NativeG$ is equivalent to the
generalized Sobolev space $\Hp$, i.e.,
\[
(f,g)_{\NativeG}=(f,g)_{\Hp}, \quad{ }f,g\in\NativeG=\Hp.
\]
\end{theorem}
%
\begin{proof}
By the above discussion, the distributional Fourier transform
$\hat{l}$ of $L$ is equal to $\hat{l}(\vx)=\norm{\FvP(\vx)}_2^2$.
Since $\FvP$ is nonzero on $\Rd$ and
$\vx\mapsto\norm{\FvP(\vx)}_2^{-2}\in\Lone$, $\hat{l}$ is positive
on $\Rd$ and $\hat{l}^{-1}\in\Lone$. According to
Theorem~\ref{t:GF-PD}, $G$ is a positive definite function and its
$\Lone$-Fourier transform is given by
\[
\hat{\MATHCAL{G}}(\vx):=(2\pi)^{-d/2}\hat{l}(\vx)^{-1}=(2\pi)^{-d/2}\norm{\FvP(\vx)}_2^{-2},\quad{}\vx\in\Rd.
\]
With the material developed thus far we are able to construct the
reproducing-kernel Hilbert space $\NativeG$ related to $G$. We
remark that since $\hat{l}$ is positive on $\Rd$, $Lf=0$ for some
$f\in\Ltwo$ if and only if $f=0$. This implies that $\vP f=0$ if and
only if $f=0$. Hence we can conclude that $\Hp$ is an inner-product
space under this condition.

Next, we fix any $f\in\NativeG$. According to
Theorem~\ref{t:RKHS-FT}, $f\in\RealContinue\cap\Ltwo$ possesses an
$\Ltwo$-Fourier transform $\hat{f}$ and
$\vx\mapsto\hat{f}(\vx)\norm{\FvP(\vx)}_2\in\Ltwo$. This means that
the functions $\hat{p}_j\hat{f}\in\Ltwo$, $j=1,\ldots,n$. Therefore we can
define
\[
f_{P_j}:=(\hat{p}_j\hat{f})\check{}\in\Ltwo,\quad{}j=1,\ldots,n,
\]
using the inverse $\Ltwo$-Fourier transform.

Since $\hat{p}_j$ is a polynomial for each $j=1,\ldots,n$,
$\hat{p}_j\check{\overline{\gamma}}\in\Schwartz$ for each
$\gamma\in\Schwartz$. Moreover,
$\hat{p}_j\check{\overline{\gamma}}=\hat{p}_j\overline{\hat{\gamma}}
=\overline{\hat{p}_j^{\ast}\hat{\gamma}}
=\overline{\widehat{P_j^{\ast}\gamma}}$ implies that
\begin{eqnarray*}
(f_{P_j},\overline{\gamma})=((\hat{p}_j\hat{f})\check{},\overline{\gamma})
=(\hat{p}_j\hat{f},\check{\overline{\gamma}})
=\langle\hat{f},\hat{p}_j\check{\overline{\gamma}}\rangle =\langle
\hat{f},\overline{\widehat{P_j^{\ast}\gamma}}\rangle =\langle
f,\overline{P_j^{\ast}\gamma}\rangle  =\langle
f,P_j^{\ast}\overline{\gamma}\rangle =\langle
P_jf,\overline{\gamma}\rangle.
\end{eqnarray*}
This shows that $P_jf=f_{P_j}\in\Ltwo$. Therefore we know that
$f\in\Hp$ and then $\NativeG\subseteq\Hp$.

To establish equality of the inner products we let $f,g\in\NativeG$.
Then the Plancherel theorem~\cite{HunNac05} yields
\begin{eqnarray*}
(f,g)_{\Hp}&=&\sum_{j=1}^{n}(f_{P_j},\overline{g_{P_j}})
=\sum_{j=1}^{n}(\hat{p}_j\hat{f},\overline{\hat{p}_j\hat{g}})
=\sum_{j=1}^{n}\int_{\Rd}\hat{f}(\vx)\overline{\hat{g}(\vx)}\abs{\hat{p}_j(\vx)}^2\ud\vx\\
&=&\int_{\Rd}\hat{f}(\vx)\overline{\hat{g}(\vx)}\norm{\FvP(\vx)}_2^2\ud\vx
=\int_{\Rd}\hat{f}(\vx)\overline{\hat{g}(\vx)}\hat{l}(\vx)\ud\vx\\
&=&(2\pi)^{-d/2}\int_{\Rd}\frac{\hat{f}(\vx)\overline{\hat{g}(\vx)}}{\hat{\MATHCAL{G}}(\vx)}\ud\vx
=(f,g)_{\NativeG}.
\end{eqnarray*}

Finally, we verify that $\NativeG=\Hp$. We fix any $f\in\Hp$. Let
$\hat{f}$ and $\widehat{P_jf}$, respectively, be the $\Ltwo$-Fourier
transforms of $f$ and $P_jf$, $j=1,\ldots,n$. Using the Plancherel
theorem \cite{HunNac05} again we obtain
\[
\int_{\Rd}\abs{\hat{f}(\vx)\hat{p}_j(\vx)}^2\ud\vx=(\hat{p}_j\hat{f},\overline{\hat{p}_j\hat{f}})
=(\widehat{P_jf},\overline{\widehat{P_jf}})=(P_jf,\overline{P_jf})<\infty.
\]
And therefore, with the help of the proof above, we have
\begin{eqnarray*}
\int_{\Rd}\frac{\abs{\hat{f}(\vx)}^2}{\hat{\MATHCAL{G}}(\vx)}\ud\vx
&=&(2\pi)^{d/2}\int_{\Rd}\abs{\hat{f}(\vx)}^2\hat{l}(\vx)\ud\vx
=(2\pi)^{d/2}\int_{\Rd}\abs{\hat{f}(\vx)}^2\norm{\FvP(\vx)}_2^2\ud\vx\\
&=&(2\pi)^{d/2}\sum_{j=1}^n\int_{\Rd}\abs{\hat{f}(\vx)\hat{p}_j(\vx)}^2\ud\vx
<\infty
\end{eqnarray*}
showing that $\hat{\MATHCAL{G}}^{-1/2}\hat{f}\in\Ltwo$. This means
in particular that $\hat{f}\in\Lone$ because
\[
\int_{\Rd}\abs{\hat{f}(\vx)}\ud\vx\leq
\left(\int_{\Rd}\frac{\abs{\hat{f}(\vx)}^2}{\hat{\MATHCAL{G}}(\vx)}\ud\vx\right)^{1/2}
\left(\int_{\Rd}\hat{\MATHCAL{G}}(\vx)\right)^{1/2}<\infty.
\]
Since the inverse $\Lone$-Fourier transform of $\hat{f}$ is
continuous, $f\in\Continue$. According to Theorem~\ref{t:PD-FT},
$f\in\NativeG$ which implies that $\Hp\subseteq\NativeG$.

\end{proof}

\section{Examples of Positive Definite Kernels Generated By Green Functions}\label{s:examples}
\subsection{One-Dimensional Cases}

With the theory we developed in the preceding section in mind we again discuss the univariate Sobolev splines of Example~\ref{ex:unisobolevspline}.

\begin{example}[Univariate Sobolev Splines]\label{ex:Sobolevspline}
Let $\sigma>0$ be a scalar parameter and consider $\vP:=\left( d/d
\MATHRM{x},\sigma I\right)^T$ with $L:=\vP^{\ast T}\vP=\sigma^2I-d^2/d \MATHRM{x}^2$.
It is known that the Green function with respect to $L$ is
\[
G(x):=\frac{1}{2\sigma}\exp(-\sigma\abs{\MATHRM{x}}),\quad{}\MATHRM{x}\in\RR.
\]
Since $\vP$ and $G$ satisfy the conditions of Theorem~\ref{t:RKHS-Hp},
we know $G$ is positive definite and that the interpolant formed by
$G$ is given by
\begin{equation}\label{Sobolevspline}
s_{f,X}(\MATHRM{x}):=\sum_{j=1}^N c_jG(\MATHRM{x}-\MATHRM{x}_j),\quad{}\MATHRM{x}\in\RR.
\end{equation}
Formula~(\ref{Sobolevspline}) denotes a Sobolev spline (or Mat\'ern)
interpolant.

Since $f\in\HP(\RR)$ if and only if $f',f\in\Ltwo$, $\Hil^1(\RR)$ is
equivalent to $\HP(\RR)$. Applying Theorem~\ref{t:RKHS-Hp} and
\cite[Theorem~13.2]{Wen05} we confirm that
\[
\norm{s_{f,X}}_{\HP(\RR)}=\min\left\{
\norm{f}_{\HP(\RR)}:f\in\MATHCAL{H}^1(\RR)\TEXT{ and
}f(\MATHRM{x}_j)=y_j,j=1,\ldots,N\right\},
\]
i.e., $s_{f,X}$ is the minimum $\HP(\RR)$-norm interpolant to the
data from $\Hil^1(\RR)$.
\end{example}

%
%

\subsection{d-Dimensional Cases}
\begin{example}[Sobolev Splines]\label{ex:Matern}
This is a generalization of Example~\ref{ex:Sobolevspline}. Let
$\vP:=\left( \vQ_0^T,\cdots,\vQ_n^T \right)^T$ with a scalar
parameter $\sigma>0$, where
\[
\vQ_j:=\left\{
\begin{array}{ll}
\left( \frac{n!\sigma^{2n-2j}}{j!(n-j)!} \right)^{1/2}\Delta^k & \TEXT{when }j=2k,\\
\left( \frac{n!\sigma^{2n-2j}}{j!(n-j)!} \right)^{1/2}\Delta^k\nabla
& \TEXT{when }j=2k+1,
\end{array}
\right. \quad{} k\in\NN_0,\ j=0,1,\ldots,n,\ n>d/2.
\]
Here we use $\Delta^0:=I$. With these definitions we get $L:=\vP^{\ast
T}\vP=(\sigma^2I-\Delta)^n$.

The \emph{Sobolev spline} (or Mat\'ern function) is known to be the
Green function with respect to $L$ (see
\cite[Chapter~6.1.6]{BerThAg04}), i.e.,
\[
G(\vx):=\frac{2^{1-n-d/2}}{\pi^{d/2}\Gamma(n)\sigma^{2n-d}}\left(\sigma\norm{\vx}_2\right)^{n-d/2}K_{d/2-n}\left(\sigma\norm{\vx}_2\right),
\quad{}\vx\in\Rd,
\]
where $K_{m}(\cdot)$ is the \emph{modified Bessel function of the
second kind of order $m$}. Since $\vP$ and $G$ satisfy the
conditions of Theorem~\ref{t:RKHS-Hp}, $G$ is positive definite and
the associated interpolant $s_{f,X}$ is the same as the Sobolev
spline (or Mat\'ern) interpolant.

Since $f\in\Hp$ if and only if $\Delta^{n/2}f,f\in\Ltwo$, $\Hp$ is
equivalent to $\Hil^n(\Rd)$ which implies that $\NativeG$ and
$\Hil^n(\Rd)$ are isomorphic by Theorem~\ref{t:RKHS-Hp}. It follows
that the real classical Sobolev space $\Hil^n(\Rd)$ becomes a
reproducing-kernel Hilbert space with $\Hp$-inner-product and its
reproducing kernel is $K(\vx,\vy):=G(\vx-\vy)$.
\end{example}

In the following example we are not able to establish that the operator $\vP$ satisfies the conditions of Theorem~\ref{t:RKHS-Hp} and so part of the connection to the theory developed in this paper is lost. We therefore use the symbol $\Phi$ to denote the kernel instead of $G$.

\begin{example}[Gaussians]\label{ex:Gaussian}
The Gaussian kernel $K(\vx,\vy):=\Phi(\vx-\vy)$ based on the
Gaussian function $\Phi$ is very important and popular in current
research fields such as scattered data approximation and machine
learning. Many people would therefore like to better understand the
reproducing-kernel Hilbert space associated with the Gaussian
function. In this example, we will show that the reproducing-kernel
Hilbert space of the Gaussian function is equivalent to a
generalized Sobolev space.

We first consider the \emph{Gaussian} function
\[
\Phi(\vx):=\frac{\sigma^d}{\pi^{d/2}}\exp(-\sigma^2\norm{\vx}_2^2),\quad{}\vx\in\Rd,
\quad{}\sigma>0
\]
We know that $\Phi$ is a positive definite function and its
$\Lone$-Fourier transform is given by (see \cite[Chapter~4]{Fas07})
\[
\hat{\phi}(\vx)=(2\pi)^{-d/2}\exp\left(
-\frac{\norm{\vx}_2^2}{4\sigma^2} \right),\quad{}\vx\in\Rd.
\]
According to Theorem~\ref{t:RKHS-FT} the reproducing-kernel Hilbert
space of $\Phi$ is given by
\[
\NativePhi=\left\{
f\in\Continue\cap\RealLtwo:\hat{\phi}^{-1/2}\hat{f}\in\Ltwo
\right\},
\]
and its inner-product is equal to
\[
(f,g)_{\NativePhi}=(2\pi)^{-d/2}\int_{\Rd}\frac{\hat{f}(\vx)\overline{\hat{g}(\vx)}}{\hat{\phi}(\vx)}\ud\vx,
\quad{}f,g\in\NativePhi,
\]
where $\hat{f},\hat{g}$ are the $\Ltwo$-Fourier transforms of
$f,g\in\NativePhi$, respectively.

Let $\vP:=\left( \vQ_0^T,\cdots,\vQ_n^T,\cdots \right)^T$, where
\begin{equation}\label{P_Gaussian}
\vQ_n:=\left\{
\begin{array}{ll}
\left( \frac{1}{n!4^n\sigma^{2n}} \right)^{1/2}\Delta^k & \TEXT{when }n=2k,\\
\left( \frac{1}{n!4^n\sigma^{2n}} \right)^{1/2}\Delta^k\nabla &
\TEXT{when }n=2k+1,
\end{array}
\right.\quad{} k\in\NN_0.
\end{equation}
Here we again use $\Delta^0:=I$. Now we will verify that
$\NativePhi$ is equivalent to $\Hp$. Even though we find that $\vP$
does \emph{not} satisfy the conditions of Theorem~\ref{t:RKHS-Hp},
we are still able to use other techniques in order to combine the results of
this paper to complete the proof.

Let $\vP_n:=\left( \vQ_0^T,\cdots,\vQ_n^T \right)^T$ and
$L_n:=\vP_n^{\ast T}\vP_n$ for each $n\in\NN$. We choose the Green
function $G_n$ with respect to $L_n$, which is the inverse
$\Lone$-Fourier transform of $(2\pi)^{-d/2}\hat{l}_n^{-1}$ when
$n>d/2$. Therefore $\vP_n$ and $G_n$ satisfy the conditions of
Theorem~\ref{t:RKHS-Hp}. This tells us that -- as in
Examples~\ref{ex:Sobolevspline} and~\ref{ex:Matern} --
$\mathrm{H}_{\vP_n}(\Rd)$ is equivalent to the classical Sobolev
space $\Hil^{n}(\Rd)$ for each $n\in\NN$. Theorem~\ref{t:RKHS-Hp}
further tells us that
\[
\Native_{G_n}(\Rd)\equiv\mathrm{H}_{\vP_n}(\Rd),\quad{}\TEXT{when
}n>d/2.
\]
Furthermore, we can verify that
\[
f\in\Hp \quad\Longleftrightarrow\quad
f\in\cap_{n=1}^{\infty}\mathrm{H}_{\vP_n}(\Rd)\TEXT{ and }
\sup_{n\in\NN}\norm{f}_{\mathrm{H}_{\vP_n}(\Rd)}<\infty,
\]
which implies that
$\norm{f}_{\mathrm{H}_{\vP_n}(\Rd)}\to\norm{f}_{\Hp}$ as
$n\to\infty$.

Let $f\in\Native_{\Phi}(\Rd)$ and $\hat{f}$ be the $\Ltwo$-Fourier
transform of $f$. We can check that
$\norm{\FvP_1(\vx)}_2^2\leq\cdots\leq\norm{\FvP_n(\vx)}_2^2
\leq\cdots\leq(2\pi)^{-d/2}\hat{\phi}(\vx)^{-1}$ and
$\norm{\FvP_n(\vx)}_2^2\to(2\pi)^{-d/2}\hat{\phi}(\vx)^{-1}$ as
$n\to\infty$. Hence, $\hat{\MATHCAL{G}}_n^{-1/2}\hat{f}\in\Ltwo$
which implies that $f\in\Native_{G_n}(\Rd)$ by
Theorem~\ref{t:RKHS-FT}. According to the Lebesgue monotone
convergence theorem~\cite{AdaFou03} and Theorem~\ref{t:RKHS-Hp}, we
have
\begin{eqnarray*}
&&\lim_{n\to\infty}\norm{f}_{\mathrm{H}_{\vP_n}(\Rd)}^2
=\lim_{n\to\infty}\norm{f}_{\Native_{G_n}(\Rd)}^2
=\lim_{n\to\infty}(2\pi)^{-d/2}\int_{\Rd}\frac{\abs{\hat{f}(\vx)}^2}{\hat{\MATHCAL{G}}_n(\vx)}\ud\vx\\
&=&\lim_{n\to\infty}\int_{\Rd}\abs{\hat{f}(\vx)}^2\hat{l}_n(\vx)\ud\vx
=\lim_{n\to\infty}\int_{\Rd}\abs{\hat{f}(\vx)}^2\norm{\FvP_n(\vx)}_2^2\ud\vx\\
&=&(2\pi)^{-d/2}\int_{\Rd}\frac{\abs{\hat{f}(\vx)}^2}{\hat{\phi}(\vx)}\ud\vx
=\norm{f}_{\NativePhi}^2<\infty.
\end{eqnarray*}
Therefore $f\in\Hp$ and $\norm{f}_{\NativePhi}=\norm{f}_{\Hp}$.

On the other hand, we fix any $f\in\Hp$. Then
$f\in\mathrm{H}_{\vP_n}(\Rd)$ for each $n\in\NN$. We again use the
Lebesgue monotone convergence theorem~\cite{AdaFou03} and
Theorem~\ref{t:RKHS-Hp} to show that
\begin{eqnarray*}
\int_{\Rd}\frac{\abs{\hat{f}(\vx)}^2}{\hat{\phi}(\vx)}\ud\vx
&=&(2\pi)^{d/2}\lim_{n\to\infty}\int_{\Rd}\abs{\hat{f}(\vx)}^2\norm{\FvP_n(\vx)}_2^2\ud\vx
=(2\pi)^{d/2}\lim_{n\to\infty}\int_{\Rd}\abs{\hat{f}(\vx)}^2\hat{l}_n(\vx)\ud\vx\\
&=&\lim_{n\to\infty}\int_{\Rd}\frac{\abs{\hat{f}(\vx)}^2}{\hat{\MATHCAL{G}}_n(\vx)}\ud\vx
=(2\pi)^{d/2}\lim_{n\to\infty}\norm{f}_{\Native_{G_n}(\Rd)}^2 \\
&=&(2\pi)^{d/2}\lim_{n\to\infty}\norm{f}_{\mathrm{H}_{\vP_n}(\Rd)}^2=(2\pi)^{d/2}\norm{f}_{\Hp}^2<\infty,
\end{eqnarray*}
which establishes that $\hat{\phi}^{-1/2}\hat{f}\in\Ltwo$, and therefore $f\in\NativePhi$.

Summarizing the above discussion, it follows that the
reproducing-kernel Hilbert space of the Gaussian kernel is given by
the generalized Sobolev space $\Hp$, i.e.,
\[
\NativePhi\equiv\Hp.
\]
\end{example}
%
\begin{remark}\label{Gaussian_Remark}
If $f\in\NativePhi\equiv\Hp$, then $f\in\Hil^{n}(\Rd)$ for each
$n\in\NN$. According to the Sobolev embedding
theorem~\cite{AdaFou03}, we have $\NativePhi\subseteq
\Cont^{\infty}_b(\Rd)$. On the other hand, if a function
$f\in\Real(\Cont^{\infty}_b(\Rd))$ satisfies
$\norm{D^{\alpha}f}_{\Leb_{\infty}(\Rd)}\leq C^{\abs{\alpha}}$ for
some positive constant $C$ and each $\alpha\in\NN_0^d$, then
$f\in\Hp\equiv\NativePhi$. Moreover, if we replace the test
functions space to be $\MATHSCR{D}(\Rd)$, then we can further think
of the Gaussian function $\Phi$ is a (full-space) Green function of
$L:=\exp(-\frac{1}{4\sigma^2}\Delta)$, i.e.,
\[
L\Phi=\delta_0\quad{}\text{and}\quad{}\Phi,\delta_0\in\MATHSCR{D}(\Rd)'.
\]
\end{remark}

%
%
\section{Extensions and Future Works}\label{s:closing}

In another paper~\cite{FasYe10}, we generalize the results of this
paper in several ways.

Instead of being limited to differential operators we allow general
distributional operators which, e.g., are allowed to be differential
operators with non-constant coefficients. In that case the
(full-space) Green functions and the generalized Sobolev spaces can
be constructed by the vector distributional operators in a similar
way.

In addition, we extend all the results from positive definite
functions and their reproducing-kernel Hilbert space with respect to
a vector differential operators to conditionally positive definite
functions of some orders and their native space with respect to
vector distributional operators. In this case the real generalized
Sobolev space will be rewritten as the following from
\[
\Hp:= \Big\{ f\in\RealLloc\cap\SI:\left\{P_n
f\right\}_{n=1}^{\infty}\subseteq\Ltwo\TEXT{ and }
\sum_{n=1}^{\infty}\norm{P_nf}_{\Ltwo}^2<\infty \Big\}
\]
and it is equipped with the semi-inner product
\[
(f,g)_{\Hp}:=\sum_{n=1}^{\infty}(P_n f,\overline{P_n g}), \quad{
}f,g\in\Hp,
\]
where $\vP:=\left(P_1,\cdots,P_n,\cdots\right)^T$ is a vector
distributional operator (see~\cite{FasYe10}).

For example, if $\vP:=\left(\omega_{\tau}\partial^m/\partial x_1^m,
\cdots, \omega_{\tau} D^{\alpha}, \cdots,
\omega_{\tau}\partial^m/\partial x_d^m \right)^T$ which is set up by
the differential operators (with the non-constant coefficients),
then
\[
\Hp:=\left\{f\in\RealLloc\cap\SI:\omega_{\tau}
D^{\alpha}f\in\Ltwo,\abs{\alpha}= m,\alpha\in\NN_0^d\right\},
\]
where $\omega_{\tau}(\vx):=\norm{\vx}_2^{\tau}$ and $0\leq\tau<1$.

In the work presented here and in \cite{FasYe10} we do not specify
any boundary conditions for the Green functions. Thus there may be
many different choices for the Green function with respect to one
and the same differential operator $L$. In our future work we will
apply a vector differential operator
$\vP:=\left(P_1,\cdots,P_n\right)^T$ and a vector boundary operator
$\vB:=(B_1,\cdots,B_s)^T$ on a bounded domain $\Omega$ to construct
a reproducing kernel and its related reproducing-kernel Hilbert
space (see~\cite{FasYe10b}). We further hope to use the
distributional operator $L$ to approximate the eigenvalues and
eigenfunctions of the kernel function with the hope of obtaining
fast numerical methods to solve the interpolating
systems~(\ref{interpolant})-(\ref{interpolation_conditions}) similar
as in \cite[Chapter 15]{Wen05}.

Finally, we only consider real-valued functions for the definition
of our generalized Sobolev spaces and their Green functions in this
paper. However, all conclusions and all theorems can be extended to
complex-valued functions similar as was done in \cite{Wen05}.



Thank you to Dr. G. E. Fasshauer, my advisor, who gave me guidance
and support throughout the entire paper process.

\end{document}